\documentclass[12pt]{amsart}
\usepackage[hmarginratio=1:1]{geometry} 
\usepackage{graphicx}
\usepackage{amssymb}

\newcommand{\hidden}[1]{}

\usepackage[numbers]{natbib}

\usepackage{amsmath}
\usepackage{amscd}
\usepackage{amsfonts}
\usepackage{amssymb}
\usepackage{amsthm}
\usepackage{cases}		 
\usepackage{cutwin}
\usepackage{enumerate}
\usepackage{graphicx}
\usepackage{hyperref}
\usepackage{ifthen}
\usepackage{lipsum}
\usepackage{mathrsfs}	
\usepackage{mathtools}
\usepackage{multimedia}
\usepackage{url}
\usepackage{wrapfig}
\usepackage{xcolor}

\newcommand{\hide}[1]{}
\newcommand*{\bms}[1]{{\boldsymbol{#1}}}
\newcommand*{\bm}[1]{{\bf #1}}
\newcommand*{\A}{\bm{A}}
\newcommand*{\C}{\bm{C}}
\newcommand*{\X}{\bm{X}}
\newcommand*{\Y}{\bm{Y}}
\newcommand*{\spr}{r}
\newcommand*{\x}{\bm{x}}
\newcommand*{\cv}{\bm{c}}
\newcommand*{\0}{{\bm 0}}
\newcommand*{\Q}{{\bm Q}}

\newcommand*{\tr}{^{\!\top}}
\newcommand*{\B}{\bm{B}}

\newcommand*{\F}{\bm{F}}
\newcommand*{\K}{\bm{K}}
\newcommand*{\Lam}{\bms{\Lambda}}

\newcommand*{\Pfrac}[2]{\left(\frac{#1}{#2}\right)}

\newcommand*{\Complex}{\mathbb{C}}
\newcommand{\set}[1]{\left\{ \, #1 \, \right\} }

\newcommand*{\suchthat}{\colon}
\newcommand*{\I}{\bm{I}}
\newcommand{\an}[1]{\begin{align}#1\end{align}}	
\newcommand{\ab}[1]{\begin{align*}#1\end{align*}}	

\newcommand*{\eqdef}{\, {:=}\, }
\newcommand*{\Reals}{\mathbb{R}} 

\newcommand*{\ev}{\bm{e}}

\newcommand*{\stext}{\shortintertext}	
\renewcommand{\S}{\bm{S}}
\newcommand*{\goesto}{\rightarrow}

\newcommand*{\df}[2]{\displaystyle \frac{{\rm d} #1}{{\rm d} #2}}
\newcommand*{\ddf}[2]{\displaystyle \frac{{\rm d}^2 #1}{{\rm d} #2^2}}
\newcommand*{\Pmatr}[1]{\begin{pmatrix} #1\end{pmatrix}}	
\newcommand{\Cases}[1]{\left\{ \begin{array}{ll} \displaystyle #1 \end{array}\right.}  

\newcommand*{\dfinline}[2]{{\rm d} #1 /{\rm d} #2}
\newcommand*{\ep}{\epsilon}
\newcommand*{\Order}{\mathcal{O}}

\newtheorem{Theorem}{Theorem}
\newtheorem{Proposition}[Theorem]{Proposition}
\newtheorem{Lemma}[Theorem]{Lemma}
\newtheorem{Corollary}[Theorem]{Corollary}

\renewcommand{\citet}{\cite}
\renewcommand{\citep}{\cite}

\title{
Nonconcavity of the Spectral Radius \\in Levinger's Theorem}
\author {Lee Altenberg
\\
Information and Computer Sciences, \\
University of Hawai`i at M\=anoa \\
\MakeLowercase{altenber@hawaii.edu}
\\
\\
and
\\
\\
Joel E. Cohen
\\
Laboratory of Populations, \\
Rockefeller University \& Department of Statistics, Columbia University, New York \\
Department of Statistics, University of Chicago \\
\MakeLowercase{cohen@rockefeller.edu}
}

\date{\today. \emph{Linear Algebra and Its Applications} \href {https://doi.org/10.1016/j.laa.2020.07.028} {doi.org/10.1016/j.laa.2020.07.028}
} 

\begin{document}
\begin{abstract}
Let $\A \in \Reals^{n \times n}$ be a nonnegative irreducible square matrix and let $r(\A)$ be its spectral radius and Perron-Frobenius eigenvalue.   Levinger asserted and several have proven that $r(t):=\spr((1{-}t) \A + t \A\tr)$ increases over $t \in [0,1/2]$ and decreases over $t \in [1/2,1]$.  It has further been stated that $r(t)$ is concave over $t \in (0,1)$.   Here we show that the latter claim is false in general through a number of counterexamples, but prove it is true for $\A \in \Reals^{2\times 2}$, weighted shift matrices (but not cyclic weighted shift matrices), tridiagonal Toeplitz matrices, and the 3-parameter Toeplitz matrices from Fiedler, but not Toeplitz matrices in general.  A general characterization of the range of $t$, or the class of matrices, for which the spectral radius is concave in Levinger's homotopy remains an open problem.
\end{abstract}
\maketitle
\pagestyle{myheadings}
\markboth{L. Altenberg \& J. E. Cohen}{Nonconcavity of the Spectral Radius in Levinger's Theorem}
\centerline{\emph{Dedicated to the memory of Bernard Werner Levinger (1928--2020)}}
\section{Introduction}

The variation of the spectrum of a linear operator as a function of variation in the operator has been extensively studied, but even in basic situations like a linear homotopy $(1{-}t) \X + t \Y$ between two matrices $\X, \Y$, the variational properties of the spectrum have not been fully characterized.  We focus here on Levinger's theorem about the spectral radius over the convex combinations of a nonnegative matrix and its transpose, $(1{-}t) \A + t \A\tr$.

We refer to $\B(t) = (1{-}t) \A + t \A\tr$, $t \in [0,1]$, as \emph{Levinger's homotopy},\footnote{Also called Levinger's transformation \citet{Psarrakos:and:Tsatsomeros:2003:Perron}.} and the spectral radius of Levinger's homotopy as \emph{Levinger's function} $r(t) \eqdef \spr(\B(t)) = \spr((1{-}t) \A + t \A\tr)$.

On November 6, 1969, the \emph{Notices of the American Mathematical Society} received a three-line abstract from Bernard W. Levinger for his talk at the upcoming AMS meeting, entitled ``An inequality for nonnegative matrices.''\citet{Levinger:1970:Inequality}  We reproduce it in full:

``\underline{Theorem.} Let $A \ge 0$ be a matrix with nonnegative components. Then $f(t) = p(tA + (1{-}t)A^T)$ is a monotone nondecreasing function of $t$, for $0 \le t \le 1/2$, where $p(C)$ denotes the spectral radius of the matrix $C$. This extends a theorem of Ostrowski. The case of constant $f(t)$ is discussed.''

Levinger presented his talk at the Annual Meeting of the American Mathematical Society at San Antonio in January 1970. Miroslav Fiedler and Ivo Marek were also at the meeting \citep{Marek:1974:Inequality}.  Fiedler developed an alternative proof of Levinger's theorem and communicated it to Marek \citep{Marek:1978:Perron}. Fiedler did not publish his proof until 1995 \citep{Fiedler:1995:Numerical}.  Levinger appears never to have published his proof.

Marek \citet{Marek:1978:Perron,Marek:1984:Perron} published the first proofs of Levinger's theorem, building on Fiedler's ideas to generalize it to operators on Banach spaces.  Bapat \citet{Bapat:1987:Two} proved a generalization of Levinger's theorem for finite matrices.  He showed that a necessary and sufficient condition for non-constant Levinger's function is that $\A$ have different left and right normalized (unit) eigenvectors (\emph{Perron vectors}) corresponding to the Perron-Frobenius eigenvalue (\emph{Perron root}).

Fiedler \citet{Fiedler:1995:Numerical} proved also that Levinger's function $r(t)$ is concave in some open neighborhood of $t=1/2$, and strictly concave when $\A$ has different left and right normalized Perron vectors.  The extent of this open neighborhood was not elucidated.

Bapat and Raghavan \citet[p.~121]{Bapat:and:Raghavan:1997} addressed the concavity of Levinger's function in discussing ``an inequality due to Levinger, which essentially says that for any $\A \ge 0$, the Perron root, considered as a function along the line segment joining $\A$ and $\A\tr$, is concave.'' The inference about concavity would appear to derive from the theorem of \citet[Theorem 3]{Bapat:1987:Two} that $\spr(t \, \A + (1{-}t)\B\tr) \geq t\, r(\A) + (1{-}t)\, r(\B)$ for all $t \in [0, 1]$, when $\A$ and $\B$ have a common left Perron vector and a common right Perron vector.  The same concavity conclusion with the same argument appears in \citet[Corollary 1.17]{Stanczak:Wiczanowski:and:Boche:2009:Fundamentals}.

However, concavity over the interval $t \in [0, 1]$ would require that 
for all $t, h_1, h_2 \in [0, 1]$, $r(t \, \F(h_1) + (1{-}t) \F(h_2) ) \geq t\, r(\F(h_1)) + (1{-}t) r(\F(h_2))$, 
where $\F(h) \eqdef h \, \A + (1{-} h)\B\tr$.  
While Theorem 3.3.1 of \citet{Bapat:and:Raghavan:1997} proves this for $h_1 = 1$ and $h_2 = 0$, it cannot be extended generally to $h_1, h_2 \in (0,1)$ because $\F(h_1)$ and $\F(h_2)\tr$ will not necessarily have common left eigenvectors and common right eigenvectors.  

Here, we show that the concavity claim is true for $2 \times 2$ and other special families of matrices.  We also show that for each of these matrix families, counterexamples to concavity arise among matrix classes that are ``close'' to them, in having extra or altered parameters.  Table \ref{Table:Comparison} summarizes our results.

\begin{table}[ht] \label{Table:Comparison}
\caption{Classes of nonnegative matrices with concave Levinger's function (left), and matrix classes ``close'' to them with nonconcave Levinger's function (right).}
{\small 
\begin{tabular}{|lr|lr|}
\hline
{\bf Concave} & &{\bf Nonconcave} &\\
\hline
$2 \times 2$ &\!\!\!\!Theorem \ref{Theorem:2x2}& $3 \times 3$, $4 \times 4$ &\!\!\!\!\!\!\!\!\!\!\!\!\!\!\!\!\!\!\!\!Eqs.\ \eqref{eq:Ex1}, \eqref{eq:4x4} \\
Tridiagonal Toeplitz &\!\!\!\!Theorem \ref{Theorem:Tridiag} & 4-parameter Toeplitz &Eq.\ \eqref {eq:ToeplitzConvex}\\
Fiedler's 3-parameter Toeplitz &\!\!\!\!Theorem \ref{Theorem:FiedlerLevinger} & 4-parameter Toeplitz &Eq.\ \eqref {eq:ToeplitzConvex}\\
$n \times n$ weighted shift matrix &\!\!\!\!Theorem \ref{Theorem:Shift} & $n \times n$ cyclic weighted shift matrix &Eq.\ \eqref {eq:CyclicShift16} \\
\hline
\end{tabular}
}
\end{table}

\section{Matrices that Violate Concavity}

\subsection{A Simple Example}
Let 
\an{
\A &= \Pmatr{
0&1&0\\
0&0&0\\
0&0&2/5
}\notag
\stext{to give}
\B(t) &= (1{-}t)\A+t\A\tr
= \Pmatr{
0&1{-}t&0\\
t&0&0\\
0&0&2/5}.  \label{eq:Ex1}
}
The eigenvalues of $\B(t)$ are $\{2/5, +\sqrt{t(1{-}t)}, -\sqrt{t(1{-}t)}\}$, plotted in Figure \ref{fig:Ex1}.  On the interval $t \in [1/5, 4/5]$, $\spr(\B(t)) = \sqrt{t(1{-}t)}$ is strictly concave.  On the intervals $t \in [0, 1/5]$ and $t \in [4/5,1]$, $\spr(\B(t))$ is constant.   It is clear from the figure that $\spr(\B(t))$ is not concave in the neighborhood of $t=1/5$ (and $t=4/5$), since for all small $\epsilon > 0$,
\an{\label{eq:Ineq1}
\frac{1}{2}[ \spr(\B(1/5 - \epsilon) + \spr(\B(1/5 + \epsilon) ]> \spr(\B(1/5 )) = 2/5.
}
By the continuity of the eigenvalues in the matrix elements \citep[2.4.9]{Horn:and:Johnson:2013}, we can make $\B(t)$ irreducible and yet preserve inequality \eqref{eq:Ineq1} in a neighborhood of $t=1/5$ by adding a small enough positive perturbation to each element of $\A$.
\begin{figure}[ht]
\centerline{\includegraphics[width=0.5\textwidth]{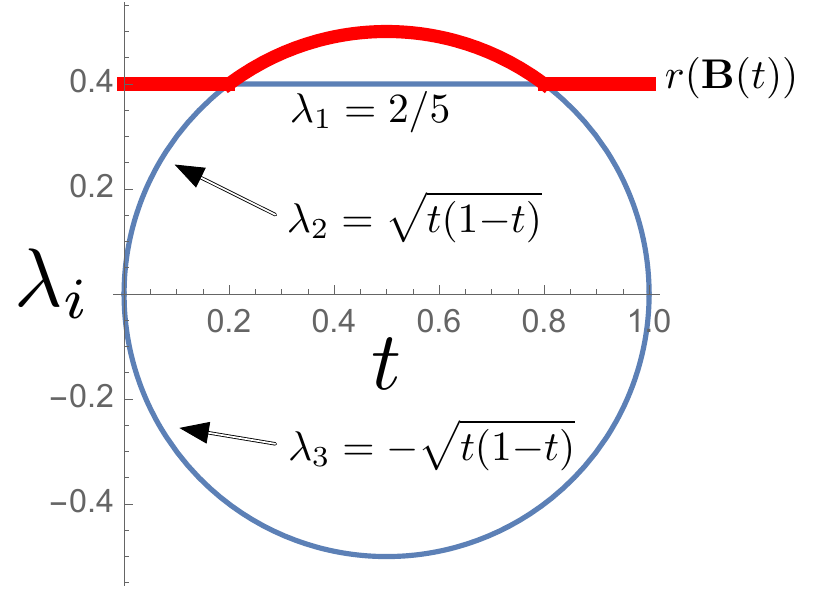}}
\caption{Eigenvalues of the matrix $\B(t)$ from \eqref{eq:Ex1}, $\lambda_1=2/5, \lambda_2= +\sqrt{t(1{-}t)}, \lambda_3= -\sqrt{t(1{-}t)}$, showing that the spectral radius $r(\B(t))$ (thick top line) is not concave around the points $t=0.2$ and $t=0.8$.}\label{fig:Ex1}
\end{figure}

The basic principle behind this counterexample is that the maximum of two concave functions need not be concave.   Here $\B(t)$ is the direct sum of two block matrices.  The eigenvalues of the direct sum are the union of the eigenvalues of the blocks, which are different functions of $t$.  One block has a constant spectral radius and the other block has a strictly concave spectral radius.  The spectral radius of $\B(t)$ is their maximum.

Another example of this principle is constructed by taking the direct sum of two $2 \times 2$ blocks, each of which is a Levinger homotopy of the matrix $\Pmatr{0&1\\0&0}$, 
but for values of $t$ at opposite ends of the {unit} interval, one $2\times 2$ block, $\A_1$, with $t_1 = 511/512$ and the other $2\times 2$ block, $\A_2$, with $t_2 =1/8$.  We take a weighted combination of the two blocks with weight $h$, $\A(h) = (1-h) \A_1 \oplus h \A_2$, to get:
\an{
\A(h) &=
\Pmatr{ 0 & (1{-}h) \frac{511}{512} & 0 & 0 \\
(1{-}h) \frac{1}{512} & 0 & 0 & 0 \\
 0 & 0 & 0 & h \frac{1}{8} \\
 0 & 0 & h\frac{7 }{8} & 0 } . \label{eq:4x4}
}
The eigenvalues of $\B(t,h) = (1{-}t) \A(h) + t \A(h)\tr$ are plotted in Figure \ref{fig:4x4}.  We see that there is a narrow region of $h$ below $h=0.5$ where the maximum eigenvalue switches from block 2 to block 1 and back to block 2 with increasing $t \in [0,1]$, making $\spr(\B(t,h)) = \spr((1{-}t) \A(h) + t \A(h)\tr)$ at $h=0.4$ nonconcave with respect to the interval $t \in [0,1] $.
\begin{figure}[th]
\centerline{\includegraphics[width=0.6\textwidth]{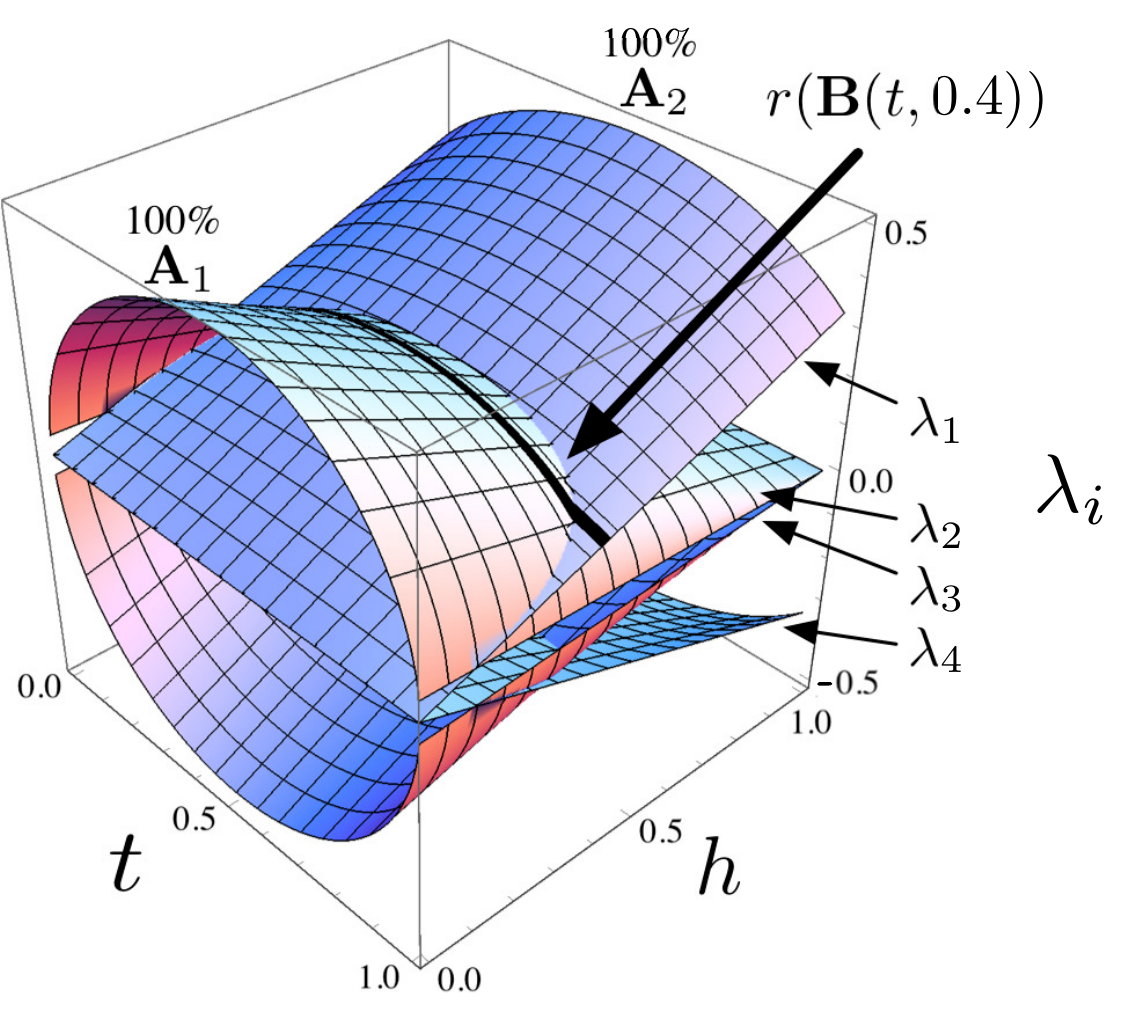}}
\caption{Eigenvalues of $\B(t,h)$ for a two-parameter homotopy: Levinger's homotopy $\B(t,h) = (1{-}t) \A(h) + t \A(h)\tr$, $t \in [0,1]$, and a second homotopy $\A(h)= (1{-}h)\A_1 \oplus h \A_2$, $h \in [0,1]$ \eqref{eq:4x4}. The dark band at $h=0.4$ is $r(\B(t,0.4))$, 
showing that the spectral radius is nonconcave in $t$ where it jumps between the two concave upper manifolds.}\label{fig:4x4}
\end{figure}
As in example \ref{eq:Ex1}, $\A(h)$ may be made irreducible by positive perturbation of the $0$ values without eliminating the nonconcavity.

The principle here may be codified as follows.

\begin{Proposition}
Let $\A = \A_1 \oplus \A_2 \in \Reals^{n \times n}$, where $\A_1$ and $\A_2$ are irreducible nonnegative square matrices.  
{Then}
$\spr(t) \eqdef \spr((1-t)\A+ t \A\tr)$ is not concave in $t \in (0,1)$ if there exists $t^* \in (0,1)$ such that 
\begin{enumerate}
\item $\spr((1-t^*)\A_1+ t^* \A_1\tr) = \spr((1-t^*)\A_2+ t^* \A_2\tr)$,\\ \ \\
{and} \ \\
\item $ \left. \df{}{t} \spr((1-t)\A_1+ t \A_1\tr)\right|_{t=t^*} 
  \neq \left. \df{}{t} \spr((1-t)\A_2+ t \A_2\tr)\right|_{t=t^*} $.
\end{enumerate}
\end{Proposition}
\begin{proof}
Let $r^* \eqdef r(t^*) = \spr((1-t^*)\A_1+ t^* \A_1\tr) = \spr((1-t^*)\A_2+ t^* \A_2\tr)$.  Since the spectral radius of a nonnegative irreducible matrix is a simple eigenvalue by Perron-Frobenius theory, it is analytic in the matrix elements \citep[Fact 1.2]{Tsing:etal:1994:Analyticity}.  Thus for each of $\A_1$ and $\A_2$, Levinger's function is analytic in $t$, and therefore has equal left and right derivatives around $t^*$.  So we can set $s_1 = \dfinline{\spr((1-t)\A_1 +  t \A_1\tr)}{t}|_{t=t^*}$ and $s_2 = \dfinline {\spr((1{-}t)\A_2+ t \A_2\tr)}{t}|_{t=t^*}$.  Then
\ab{
\spr((1{-}t^* {-} \ep)\A_1+ (t^* {+} \ep) \A_1\tr) &= r^* + \ep s_1 + \Order(\ep^2), \\
\spr((1{-}t^* {-} \ep)\A_2+ (t^* {+} \ep) \A_2\tr) &= r^* + \ep s_2 + \Order(\ep^2).
}
For a small neighborhood around $t^*$,
\ab{
\spr(t^*{+}\ep) &= \spr( (1{-}t^* {-} \ep )\A+ (t^* {+} \ep) \A\tr) \\&
= \max\set{\spr((1{-}t^* {-} \ep)\A_1+ (t^* {+} \ep) \A_1\tr), \spr((1{-}t^* {-} \ep)\A_2+ (t^* {+} \ep) \A_2\tr)} \\&
= r^* + \Cases{
\ep \min(s_1, s_2) + \Order(\ep^2),	&\qquad \ep < 0, \\
\ep \max(s_1, s_2) + \Order(\ep^2),	&\qquad \ep > 0 .
}
}
A necessary condition for concavity is 
$
\frac{1}{2}(\spr(t^*{+} \ep) + \spr(t^*{-}\ep)) \leq \spr(t^*).
$
However, for small enough $\ep > 0$, letting $\delta = \max(s_1, s_2) - \min(s_1, s_2) > 0$,
\ab{
\frac{\spr(t^*{+} \ep) + \spr(t^*{-}\ep)}{2}
&= r^* + \ep \frac{\max(s_1, s_2) - \min(s_1, s_2) }{2} + \Order(\ep^2) \\
&= r^* + \ep \delta / {2} + \Order(\ep^2)
> r^* .
}
The condition for concavity is thus violated.
\end{proof}
\subsection{Toeplitz Matrices}

The following nonnegative irreducible Toeplitz matrix has a nonconcave Levinger's function:
\an{\label{eq:ToeplitzConvex}
\A&=
\Pmatr{
 5 & 0 & 6 & 0 \\
 1 & 5 & 0 & 6 \\
 0 & 1 & 5 & 0 \\
 8 & 0 & 1 & 5 
 }
}
A plot of Levinger's function for \eqref{eq:ToeplitzConvex} is not unmistakably nonconcave, so instead we plot the second derivative of $\spr(\B(t))$ in Figure \ref{fig:ToeplitzConvex}, which is positive at the boundaries $t=0$ and $t=1$, and becomes negative in the interior.
\begin{figure}[ht]
\centerline{\includegraphics[width=0.56\textwidth]{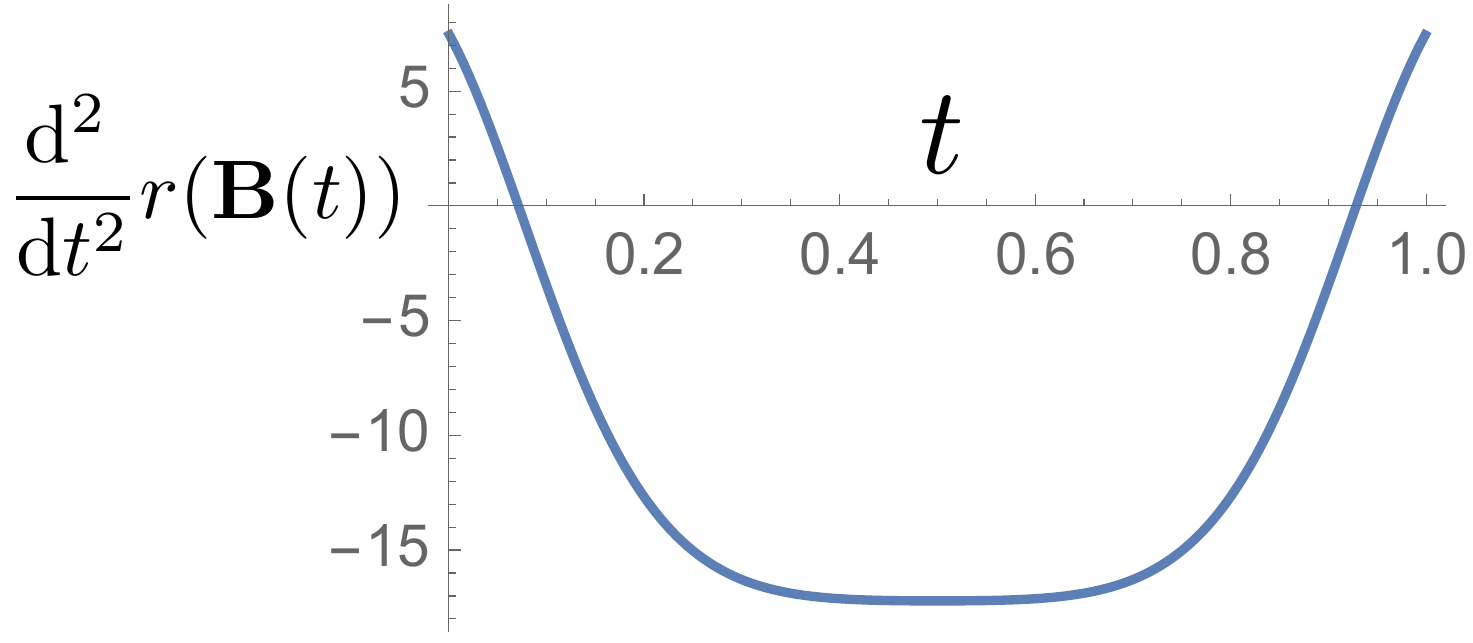}}
\caption{The second derivative of Levinger's function for the Toeplitz matrix \eqref{eq:ToeplitzConvex}.\label{fig:ToeplitzConvex}}
\end{figure}

\subsection{Weighted Circuit Matrices}

Another class of matrices where Levinger's function can be nonconcave is the weighted circuit matrix.  A weighted circuit matrix is an $n \times n$ matrix in which there are $k\in[1,n]$ distinct integers $i_1, i_2, \ldots, i_k \in \{1, 2, \ldots, n\}$ such that all elements are zero except weights $c_j$, $j = 1, \ldots, k$, at matrix positions $(i_1, i_2), (i_2, i_3), \ldots, (i_{k-1}, i_k), (i_k, i_1)$, which form a circuit.  We refer to a \emph{positive weighted circuit matrix} when the weights are all positive numbers. 

When focusing on the spectral radius of a positive weighted circuit matrix, we may without loss of generality consider its non-zero principal submatrix, whose canonical permutation of the indices gives a \emph{positive cyclic weighted shift matrix}, $\A$, with elements
\an{
A_{ij} &= \Cases{
c_i > 0,	&\qquad j = i \text{ mod } n + 1, \quad i \in \set{1, \ldots, n}, \\
0,	& \qquad\text{otherwise.}
}\label{eq:CyclicShift}
}
Equation \eqref{eq:CyclicShift} defines a cyclic \emph{downshift} matrix, while an \emph{upshift} matrix results from replacing $j = i \text{ mod } n + 1$ with $i = j \text{ mod } n + 1$, which is equivalent for our purposes.  Cyclic weighted shift matrices have the form
\ab{
\Pmatr{
 0 & c_1 & 0 & 0 \\
 0 & 0 & c_2 & 0 \\
 0 & 0 & 0 & c_3 \\
 c_4 & 0 & 0 & 0 \\
} .
}

If one of the weights $c_i$ is set to $0$, the matrix becomes a positive non-cyclic weighted shift matrix.  In Section \ref{sec:WSM}, we show that Levinger's function of a positive non-cyclic weighted shift matrix is strictly concave.  Cyclicity from a single additional positive element $c_i>0$ allows nonconcavity.

Here we provide an example of nonconcavity using a cyclic shift matrix with \emph{reversible weights}, which have been the subject of recent attention \citep{Chien:and:Nakazato:2020:Symmetry}.  Figure \ref{fig:CyclicShift16} shows Levinger's function for a $16 \times 16$ cyclic weighted shift matrix with two-pivot reversible weights
\an{
c_j &= 16 + \sin\left(2 \pi \frac{j}{16}\right), &\quad j = 1, \ldots, 16 . \label{eq:CyclicShift16}
}
Levinger's function is convex for most of the interval $t \in [0,1]$, and is concave only in the small interval around $t=1/2$.
\begin{figure}[ht]
\centerline{\includegraphics[width=0.56\textwidth]{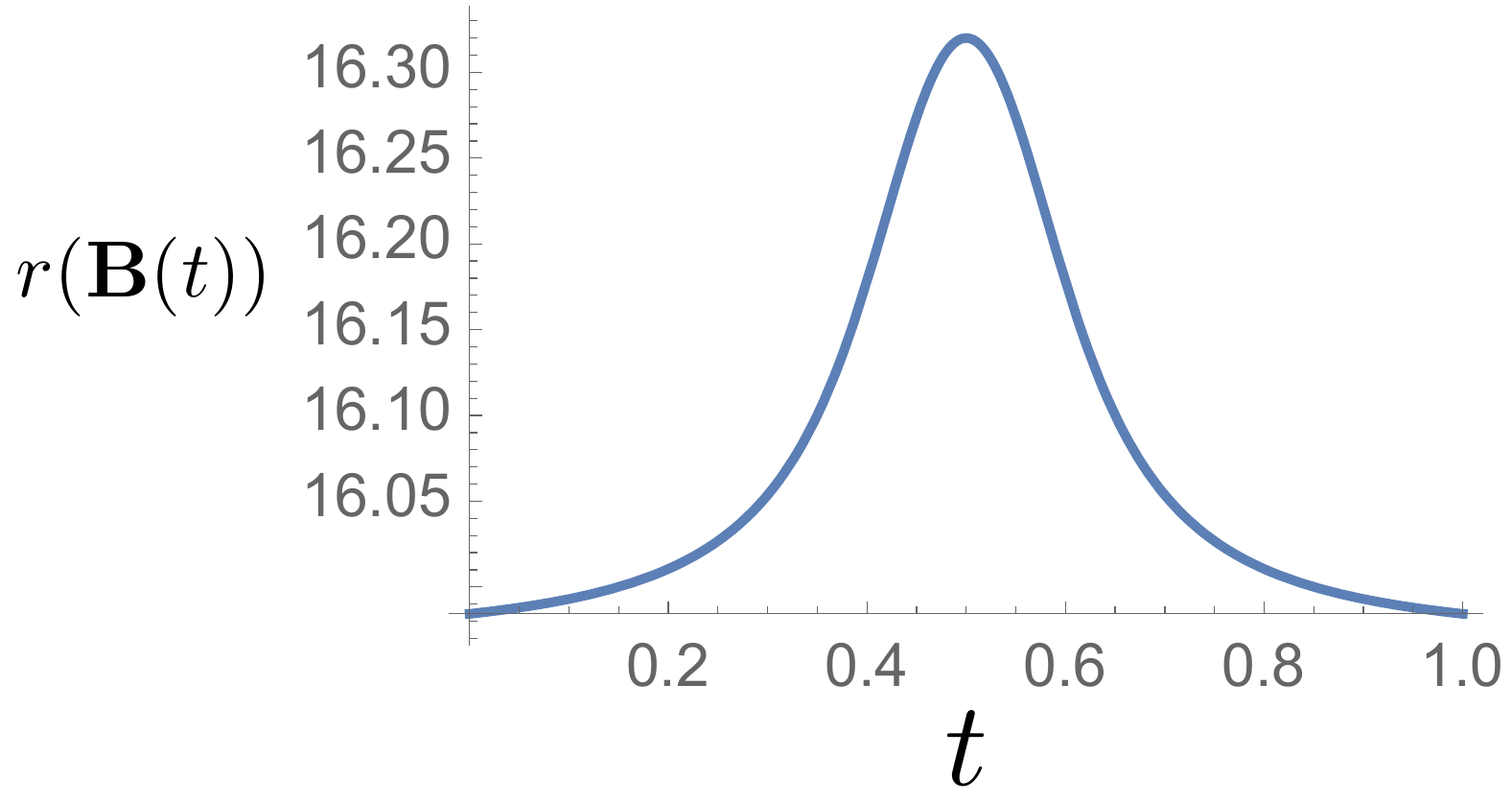}}
\caption{Nonconcave Levinger's function for a $16 \times 16$ two-pivot reversible cyclic weighted shift matrix with weights $c_j = 16 + \sin(2 \pi j /16)$, \eqref{eq:CyclicShift16}.}\label{fig:CyclicShift16} 
\end{figure}

\section{Matrices with Concave Levinger's Function}

Here we show that several special classes of nonnegative matrices have concave Levinger's functions:  $2\times 2$ matrices, non-cyclic weighted shift matrices, tridiagonal Toeplitz matrices, and Fiedler's 3-parameter Toeplitz matrices.

\subsection{\texorpdfstring{$2 \times 2$}{2 x 2}\  Matrices}

\begin{Theorem}\label{Theorem:2x2}
Let $\A \in \Reals^{2 \times 2}$ be nonnegative and irreducible. Then the spectral radius and Perron-Frobenius eigenvalue $r(t):=r((1{-}t) \A + t \A\tr)$ is concave over $t \in (0,1)$, strictly when $\A$ has different left and right Perron-Frobenius eigenvectors.
\end{Theorem}
\begin{proof}
Let $a, b,c,d \in (0,\infty), t\in (0,1)$,
and assume $b\ne c$ to assure that $\A\ne\A\tr$ and the left and right Perron-Frobenius eigenvectors are not colinear.  Let
$$ \A:=\Pmatr{
a & b\\
c & d
},
\quad \B(t):=(1{-}t) \A+t \A\tr.$$
The Perron-Frobenius eigenvalue of $\B(t)$ is obtained by using the quadratic formula to solve the characteristic equation.  After some simplification,
$$ r(t):= r(\B(t)) = \frac{a+d+\sqrt{(a-d)^2 + 4t(1{-}t)(b-c)^2 + 4bc }}{2}. $$
The first derivative with respect to $t$ is
$$r'(t)=\frac{{\left(1{-}2\,t\right)}\,{{\left(b-c\right)}}^2 }{\sqrt{(a-d)^2 + 4t(1{-}t)(b-c)^2 + 4bc}}. $$
The denominator above is positive for all $t\in(0,1)$ because of the assumption that $b\ne c$.
The second derivative is, again after some simplification,
\an{\label{eq:f''(t)}
r''(t)=-\frac{2\,(b-c)^2 \,\left((a-d)^2+(b+c)^2\right)}{{{\left((a-d)^2 + 4t(1{-}t)(b-c)^2 + 4bc\right)}}^{3/2} }<0.
}
The numerator in the fraction above is positive because $b\ne c$, and the minus sign in front of the fraction guarantees strict concavity for all $t\in(0,1)$.
\end{proof}

\subsection{Tridiagonal Toeplitz Matrices}

\begin{Theorem}[Tridiagonal Toeplitz Matrices]\label{Theorem:Tridiag}
Let $\A \in \Reals^{n \times n}$, $n \geq 2$, be a tridiagonal Toeplitz matrix with diagonal elements $b\geq 0$, subdiagonal elements $a \geq 0$, and superdiagonal elements $c \geq 0$, with $\max(a, \ c) > 0$.  Then for $t \in (0, 1)$, $\spr( (1{-}t) \A + t \A\tr )$ is concave in $t$, increasing on $t \in (0,1/2)$, and decreasing on $t \in (1/2,1)$, all strictly when $a \neq c$.
\end{Theorem}
\begin{proof}
The eigenvalues of a tridiagonal Toeplitz matrix $\A$ with $a, c \neq 0$ are \citep[22-5.18]{Hogben:2014:Handbook} \citep[Theorem 2.4]{Bottcher:and:Grudsky:2005:Spectral}
\an{
\lambda_k(\A) &= b + 2 \sqrt{a c}\ \cos\left(\frac{k \pi}{n{+}1} \right).  \label{eq:BG2005}
}

The matrix $(1{-}t) \A + t \A\tr$ has subdiagonal values $(1{-}t) a + t c$ and superdiagonal values $t a + (1{-}t) c$.  Since at least one of $a,c$ is strictly positive, $(1{-}t) a + t c > 0$ and $t a + (1{-}t) c > 0$ for $t \in (0,1)$. Therefore \eqref{eq:BG2005} is applicable.
 
Writing $\lambda_k(t) \eqdef \lambda_k( (1{-}t) \A + t \A\tr)$, we obtain
\ab{
\lambda_k(t) &= b + 2 \sqrt{((1{-}t) a + t c) (t a + (1{-}t) c)}\  \cos\left(\frac{k \pi}{n{+}1} \right).
}
It is readily verified that the first derivatives are
\ab{
\df{}{t}\lambda_k(t) &= \cos\left(\frac{k \pi}{n{+}1} \right) 
\frac{(a-c)^2 (1{-}2t)}{\sqrt{((1{-}t) a + t c) (t a + (1{-}t) c)}},
\stext{and the second derivatives are}
\ddf{}{t}\lambda_k(t) &= - \cos\left(\frac{k \pi}{n{+}1} \right) 
\frac{(a^2-c^2)^2}{2 \big[ ( (1{-}t) a + t c) (t a + (1{-}t) c )\big]^{3/2} } .
}
Since $(1{-}t) a + t c > 0$ and $t a + (1{-}t) c > 0$ for $t \in (0,1)$, the denominators are positive.  When $a=c$ both derivatives are identically zero.  When $a \neq c$, the factors not dependent on  $k$ are strictly positive for all $t \in (0,1)$ except for $t=1/2$ where the first derivative of all the eigenvalues vanishes.

Because the second derivatives have no sign changes on $t \in (0,1)$, and since $[ (1{-}t) a + t c][t a + (1{-}t) c ] > 0$, there are no inflection points.  Therefore each eigenvalue is either convex in $t$ or concave in $t$, depending on the sign of $\cos( k \pi/(n+1) )$.  The maximal eigenvalue is 
\ab{
r(t) = \lambda_1(t) = b + 2 \sqrt{((1{-}t) a + t c) (t a + (1{-}t) c)}\  \cos(\pi/(n{+}1)) .
}
From its first derivative, since $ \cos(\pi/(n{+}1)) > 0$, $r(t)$ is increasing on $t \in (0, 1/2)$ and decreasing on $t \in (1/2,1)$, strictly when $a \neq c$.  Since its second derivative is negative, $r(t)$ is concave in $t$ on $t \in (0,1)$, strictly when $a \neq c$.
\end{proof}

\subsection{Fiedler's Toeplitz Matrices}

Fiedler \citet[p. 180]{Fiedler:1995:Numerical} established this closed formula for the spectral radius of a special Toeplitz matrix.

\begin{Theorem}[Fiedler's 3-Parameter Toeplitz Matrices]\label{Theorem:Fiedler}

Consider a Toeplitz matrix $\A \in \Complex^{n \times n}$, $n \geq 3$, 
with diagonal values $(v, 0, \ldots, 0, v, w, u, 0, \ldots, 0, u)$, with $v, w, u \in \Complex$:
\an{\label{eq:FiedlerFlipped}
\A&=\Pmatr{w & u & 0 & \cdots & 0 & u\\
v & w & u & 0 & \cdots & 0\\
0 & v & w & u & \cdots & 0\\
\vdots & \vdots & \ddots & \ddots & \ddots & \vdots \\
0 & 0 & \cdots & v & w & u \\
v & 0 & \cdots & 0 & v & w
}.
}
Let $\omega = e^{2 \pi i / n}$.  The eigenvalues of $\A$ are
\ab{
\lambda_{j+1}(\A) &= w + \omega^j u^{(1{-}1/n)} v^{1/n} + \omega^{n-j} u^{1/n} v^{(1{-}1/n)}, \qquad j = 0, 1, \ldots, n{-}1.
}
\end{Theorem}

We apply Theorem \ref{Theorem:Fiedler} to the Levinger function.

\begin{Theorem}\label{Theorem:FiedlerLevinger}
Let $\A$ be defined as in \eqref{eq:FiedlerFlipped} with $u,v,w > 0$.  Then $\spr(t) \eqdef \spr( (1{-}t) \A + t \A \tr)$ is concave in $t$ for $t \in (0,1)$, strictly if $u \neq v$.
\end{Theorem}
\begin{proof}
For $u, v, w > 0$, $\spr(\A) = \lambda_1(\A) = w + u^{(1{-}1/n)} v^{1/n} + u^{1/n} v^{(1{-}1/n)}$ from Theorem \ref{Theorem:Fiedler}.

Let $\B(t) = (1{-}t) \A + t \A\tr$.  Then $\B(t)$ is again a Toeplitz matrix of the form \eqref {eq:FiedlerFlipped}, with diagonal values $(1{-}t)v{+}t u, 0, \ldots, 0, (1{-}t)v{+}t u, w, (1{-}t)u {+} t v,$ $0, \ldots, 0$, $(1{-}t)u{+}t v$ for matrix elements $A_{i, i{+}m}$, with $m \in \set{1{-}n, n{-}1}$, and $i \in $$\{\max(1, 1{-}m)$, $\ldots$, $\min(n, n{-}m)$$\}$.   So again by Theorem \ref{Theorem:Fiedler},
\ab{
\spr(\B(t)) = w &+ [(1{-}t)u + t v]^{(1{-}1/n)} \ [(1{-}t)v+t u]^{1/n} \\ &
 + [(1{-}t)u + t v]^{1/n}\  [(1{-}t)v+t u]^{(1{-}1/n)} .
}
It is readily verified that 
\ab{
&\ddf{}{t} \spr(\B(t)) \\ 
&= - \frac{n-1}{n^2 u^2 v^2} (u-v)^2 (u+v)^2 \\& 
\quad \times
 \left([(1{-}t)v + u]^{1/n} [(1{-}t)u + t v]^{(1{-}1/n)} + [(1{-}t)v + u]^{(1{-}1/n)} [(1{-}t)u + t v]^{1/n} \right) \\
 & \leq 0,
}
with equality if and only if $u = v$.
\end{proof}
With the simple exchange of $A_{1n}$ and $A_{n1}$ in \eqref {eq:FiedlerFlipped}, $\A$ would become a circulant matrix, which has left and right Perron vectors colinear with the vector of all ones, $\ev$, and would therefore have a constant Levinger's function.

\subsection{Weighted Shift Matrices} \label{sec:WSM}
An $n \times n$ weighted shift matrix, $\A$, has the form 
\ab{
A_{ij} = \Cases{
c_i,	&\qquad j=i+1, \quad i \in \set{1, \ldots, n-1},\\
0,	& \qquad \text{otherwise},
}
}
where $c_i$ are the weights.  It is obtained from a cyclic shift matrix be setting any one of the weights to $0$ and appropriately permuting the indices.  Unless we explicitly use ``cyclic'', we mean \emph{non-cyclic shift matrix} when we write ``shift matrix''.

We will show that Levinger's function for positive weighted shift matrices is strictly concave.  First we develop some lemmas. 

\begin{Lemma}\label{Lemma:poly}
Let $\cv \in \Complex^{n+1}$ be a vector of complex numbers and $\alpha \in \Complex$, $\alpha \neq 0$.  Then the roots of a polynomial $p(x) = \sum_{k=0}^n x^k \alpha^{n-k} c_k$ are $r_j = \alpha f_j(\cv)$, where $f_j: \Complex^{n+1} \goesto \Complex$, $j = 1, \ldots, n$.
\end{Lemma}
\begin{proof}
We factor and apply the Fundamental Theorem of Algebra to obtain
\ab{
p(x) &
= \sum_{k=0}^n x^k \alpha^{n-k} c_k
= \alpha^n \sum_{k=0}^n \Pfrac{x}{\alpha}^k c_k
= \alpha^n \prod_{j=1}^n \left( \frac{x}{\alpha} - f_j(\cv) \right). 
}
Hence the roots of $p(x)$ are $\set{\alpha f_j(\cv)\ |\  j = 1, \ldots, n}$.
\end{proof}

\begin{Lemma}\label{Lemma:ab}
Let $\alpha, \beta \in \Complex \backslash 0$, $\A(\alpha, \beta) = [A_{ij}]$ be a hollow tridiagonal matrix, where $A_{ij} > 0$ for $j=i+1$ and $j=i-1$, $A_{ij}=0$ otherwise, and
\ab{
A_{ij} = 
\Cases{
\alpha \, c_{ij},	& \qquad j=i+1, \quad i \in \set{1, \ldots, n-1},\\
\beta \, c_{ij},	& \qquad j=i-1, \quad i \in \set{2, \ldots, n},
}
}
so $\A(\alpha, \beta) $ has the form
\ab {
\A(\alpha,\beta) &=\Pmatr{
0 & \alpha \, c_{12} & 0 & \cdots & 0 & 0 & 0\\
\beta \, c_{21} & 0 & \alpha \, c_{23}& \cdots & 0 & 0 & 0\\
0 & \beta \, c_{32} & 0 & \ddots & 0 &0 & 0\\
\vdots & \vdots & \ddots & \ddots & \ddots & \vdots & \vdots \\
0 & 0 & 0 & \ddots & 0 & \alpha \, c_ {n{-}2, n{-}1} & 0 \\
0 & 0 & 0 & \cdots & \beta \, c_{n{-}1, n{-}2} & 0 & \alpha \, c_ {n{-}1, n} \\
0 & 0 & 0 & \cdots & 0 & \beta \, c_ {n, n{-}1} & 0
}.
}
Let $\cv \in \Complex^{2(n-1)}$ represent the vector of $c_{ij}$ constants.

Then the eigenvalues of $\A$ are of the form $\sqrt{\alpha \beta}\:  f_h(\cv) $, $h = 1, \ldots, n$, where $f_h\suchthat \Complex^{2(n-2)} \goesto \Complex$ are functions of the $c_{ij}$ constants that do not depend on $\alpha$ or $\beta$.
\end{Lemma}
\begin{proof}
The characteristic polynomial of $\A$ is
\ab{
p_\A(\lambda) &=
\det(\lambda\I - \A)\\
&= 
\begin{vmatrix}
\lambda & -\alpha \, c_{12} & 0 & \cdots & 0 & 0 & 0\\
-\beta \, c_{21} & \lambda & -\alpha \, c_{23}& \cdots & 0 & 0 & 0\\
0 & -\beta \, c_{32} & \lambda & \ddots & 0 &0 & 0\\
\vdots & \vdots & \ddots & \ddots & \ddots & \vdots & \vdots \\
0 & 0 & 0 & \ddots & \lambda & -\alpha \, c_ {n{-}2, n{-}1} & 0 \\
0 & 0 & 0 & \cdots & -\beta \, c_{n{-}1, n{-}2} & \lambda & -\alpha \, c_ {n{-}1, n} \\
0 & 0 & 0 & \cdots & 0 & -\beta \, c_ {n, n{-}1} & \lambda
\end{vmatrix}.
}
The characteristic polynomial has the recurrence relation
\an{
p_{\A_k}(\lambda) &
= \lambda \  p_{\A_{k-1}}(\lambda) - \alpha \, \beta \, c_{k,k{-}1} \  c_{k{-}1,k} \  p_{\A_{k-2}}(\lambda), &\quad k \in \set{3, \ldots, n}, \label{eq:rec}\\
\stext{with initial conditions}
p_{\A_2}(\lambda) &= \lambda^2 - \alpha \, \beta \ c_{12}\ c_{21}, \quad\text{and} \label{eq:A2} \\
p_{\A_1}(\lambda) &= \lambda, \label{eq:A1}
}
where $\A_k$ is the principal submatrix of $\A$ over indices $1, \ldots, k$.

We show by induction that for all $k \in \set{2, \ldots, n}$,
\an{
p_{\A_k}(\lambda) &
= \sum_{j=0}^k \lambda^j (\alpha \beta)^{ (k-j)/2} \, g_{jk}(\cv)
= \sum_{j=0}^k \lambda^j \sqrt{\alpha \beta}^{\, (k-j)} g_{jk}(\cv),	 \label{eq:IH}
}
where each $g_{jk}\suchthat \Complex^{2(n-1)} \goesto \Complex$, $k \in \set{2, \ldots, n}$, $j \in \set{0, \ldots, k}$, is a function of constants $\cv$.

From \eqref{eq:A2}, we see that \eqref{eq:IH} holds for $k=2$:
$
p(\A_2)(\lambda) = \lambda^2 - \alpha \, \beta \,  c_{12}\, c_{21} .
$

For $k=3$, from the recurrence relation \eqref{eq:rec} and initial conditions \eqref{eq:A1}, \eqref{eq:A2}, we have
\ab{
p(\A_3)(\lambda) &
= \lambda\, p_{\A_{2}}(\lambda) - \alpha \beta \, c_{32} \, c_{23} \, p_{\A_{1}}(\lambda) 
= \lambda (\lambda^2 - \alpha \beta \, c_{12}\, c_{21}) - \alpha \beta \, c_{32} \, c_{23} \, \lambda \\ &
= \lambda^3 - \lambda \sqrt{\alpha \beta}^{\, 2}( c_{12}\, c_{21} + c_{32} \, c_{23}),
}
which satisfies \eqref{eq:IH}.  These are the basis steps for the induction.

For the inductive step, we need to show that if \eqref{eq:IH} holds for $k-1, k-2$ then it holds for $k$.  Suppose that \eqref{eq:IH} holds for $2 \leq k-1, k-2 \leq n-1$.  Then
\ab{
&p_{\A_k}(\lambda) 
= \lambda\, p_{\A_{k-1}}(\lambda) - \alpha\beta \, c_{k,k{-}1} \, c_{k{-}1,k} \  p_{\A_{k-2}}(\lambda)\\ &
= \lambda \sum_{j=0}^{k{-}1} \lambda^j \sqrt{\alpha \beta}^{\, (k{-}1{-}j)} g_{j,k{-}1}(\cv)
 - \alpha\beta \, c_{k,k{-}1} \, c_{k{-}1,k}   \sum_{j=0}^{k-2} \lambda^j \sqrt{\alpha \beta}^{\, (k-2-j)} g_{j,k-2}(\cv) \\ &
= \sum_{j=1}^{k} \lambda^{j} \sqrt{\alpha \beta}^{\,( k{-}j)} g_{j-1,k{-}1}(\cv)
 {-} \sum_{j=0}^{k{-}2} \lambda^j \sqrt{\alpha \beta}^{\, (k-j)} c_{k,k{-}1} \, c_{k{-}1,k} \ g_{j,k-2}(\cv) , 
}
which satisfies \eqref{eq:IH}.  Thus by induction $p_{\A_n}(\lambda)$ satisfies \eqref{eq:IH}.

Then Lemma \ref{Lemma:poly} implies that the parameters $\set{\alpha, \beta}$ appear 
as the linear factor $\sqrt{\alpha \beta}$
in each root of the characteristic polynomial of $\A(\alpha,\beta)$ --- its eigenvalues.
\end{proof}

\begin{Theorem}[Weighted Shift Matrices]\label{Theorem:Shift}
Levinger's function is strictly concave for nonnegative weighted shift matrices with at least one positive weight.
\end{Theorem}
\begin{proof}
Let the positive weighted shift matrix $\A$ be defined as
\ab{
A_{ij} = \Cases{
c_i \geq 0,	& \qquad j=i+1, \quad i \in \set{1, \ldots, n-1},\\
0,	& \qquad \text{otherwise},
}
}
where $c_i$ are the weights and $c_i >0$ for at least one $i = 1, \ldots, n-1$.

By Lemma \ref{Lemma:ab}, all the eigenvalues of Levinger's homotopy $\B(t) = (1{-}t) \A + t \A\tr$ are of the form $\lambda_i(\B(t)) = \sqrt{t(1{-}t)}\, f_i(\cv)$, where $\cv$ is the vector of weights, and $f_i\suchthat \Reals^{n-1} \goesto \Reals$, since $\B(t)$ is a direct sum of one or more (if some $c_i=0$) Jacobi matrices and these have real eigenvalues \citep[22.7.2]{Hogben:2014:Handbook}. 

If at least one weight $c_i$ is positive, then $\B(t)$ has a principal submatrix $\Pmatr{0& (1{-}t)  c_i\\t\,  c_i&0}$ with a positive spectral radius for $t \in (0,1)$.  Thus by \citet[Corollary 8.1.20(a)]{Horn:and:Johnson:2013}, $\spr(\B(t)) > 0$ for $t \in (0,1)$.  Therefore for $t \in (0,1)$, $\spr(\B(t)) = \lambda_1(\B(t)) = \sqrt{t(1{-}t)}\: f_1(\cv) > 0$.  Since $\sqrt{t(1{-}t)}$ is strictly concave in $t$ for $t \in (0,1)$, Levinger's function is strictly concave in $t$ for $t \in (0,1)$.
\end{proof}

\begin{Corollary}
Levinger's function is strictly concave for a nonnegative hollow tridiagonal matrix, $\A \in \Reals^{n \times n}$, in which $A_{ii}=0$ for $i \in \set{1, \ldots, n}$, and where for each $i \in \set{1, \ldots, n-1}$, $A_{i,i+1}\,  A_{i+1,i} = 0$, and for at least one $i$, $A_{i,i+1} > 0$.
\end{Corollary}
\begin{proof}
$\A$ is derived from a weighted shift matrix by swapping some elements of the superdiagonal $A_{i,i+1}$ to the transposed position in the subdiagonal, $A_{i+1,i}$.  The determinant of Levinger's homotopy $\det(\lambda \I - \B(t)) = \det(\lambda \I - (1-t) \A - t \A\tr)$ remains unchanged under such swapping because the term $\alpha \beta \, c_{k,k{-}1} \, c_{k{-}1,k}$ in \eqref{eq:rec}, which is $(1-t)t \, c_{k{-}1,k}^2$ in the weighted shift matrix, remains invariant under swapping as $t (1-t) \, c_{k,k{-}1}^2$.
\end{proof}

We complete the connection to positive weighted circuit matrices with this corollary.
\begin{Corollary}
By setting one or more, but not all, of the weights in a positive weighted circuit matrix to $0$,  Levinger's function becomes strictly concave.
\end{Corollary}
\begin{proof}
A positive weighted circuit matrix where some but not all of the positive weights are changed to $0$ is, under appropriate permutation of the indices, a nonnegative weighted shift matrix to which Theorem \ref{Theorem:Shift} applies.
\end{proof}

What kind of transition does Levinger's function make during the transition from a cyclic weighted shift matrix with nonconcave Levinger's function to a weighted shift matrix with its necessarily  concave Levinger's function, as one of the weights is lowered to $0$?  
Does the convexity observed in Figure \ref{fig:CyclicShift16} at the boundaries $t=0$ and $t=1$ flatten and become strictly concave for some positive value of that weight?  
We examine this transition for the cyclic shift matrix in example \eqref{eq:CyclicShift16} (Figure \ref{fig:CyclicShift16}).  The minimal weight is $c_{12} = 16 + \sin\left(2 \pi \frac{12}{16}\right) = 15$.  Figure \ref{fig:C12} plots Levinger's function as $c_{12}$ is divided by factors of $2^{8}$. 

Figure \ref{fig:Shift-Matrix_Limit} plots the second derivatives of Levinger's function.  We observe non-uniform convergence to the $c_{12}=0$ curve.  As $c_{12}$ decreases, the second derivative converges to the $c_{12}=0$ curve over wider and wider intervals of $t$, but outside of these intervals the second derivative \emph{diverges} from the $c_{12}=0$ curve, attaining larger values near and at the boundaries $t=0$ and $t=1$ with smaller $c_{12}$.  Meanwhile for $c_{12}=0$, Levinger's function is proportional to $\sqrt{t(1-t)}$, the second derivative of which goes to $-\infty$ as $t$ goes to $0$ or $1$.  When $c_{12}> 0$, $\B(0)$ and $\B(1)$ are irreducible, and when $c_{12}=0$, $\B(t)$ is irreducible for $t \in (0,1)$.  But for $c_{12}=0$, $\B(0)$ and $\B(1)$ are reducible matrices.   While the eigenvalues are always continuous functions of the elements of the matrix, the derivatives of the spectral radius need not be, and in this case, we see an unusual example of nonuniform convergence in the second derivative of the spectral radius.

\begin{figure}
\centerline{\includegraphics[width=0.6\textwidth]{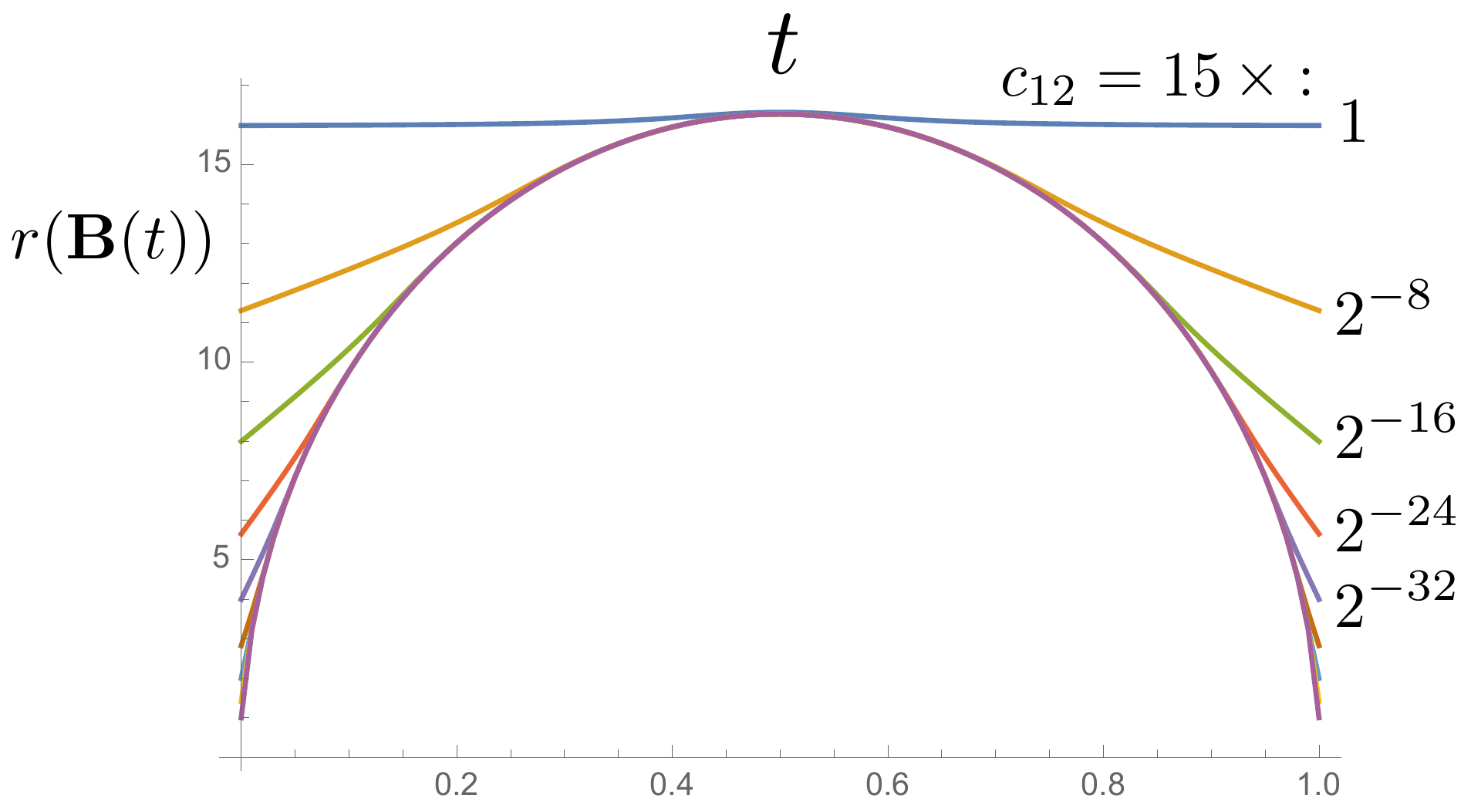}}
\caption{Levinger's function for the cyclic weighted shift matrix from \eqref{eq:CyclicShift16} in the limit as weight $c_{12}$ goes toward $0$ by being multiplied by successive powers of $2^{-8}$.  The topmost line with $c_{12} = 15 \times 1$ is the same as the curve in Figure \ref{fig:CyclicShift16} but with an expanded Y-axis.}\label{fig:C12}
\end{figure}

\begin{figure}
\centerline{\includegraphics[width=0.75\textwidth]{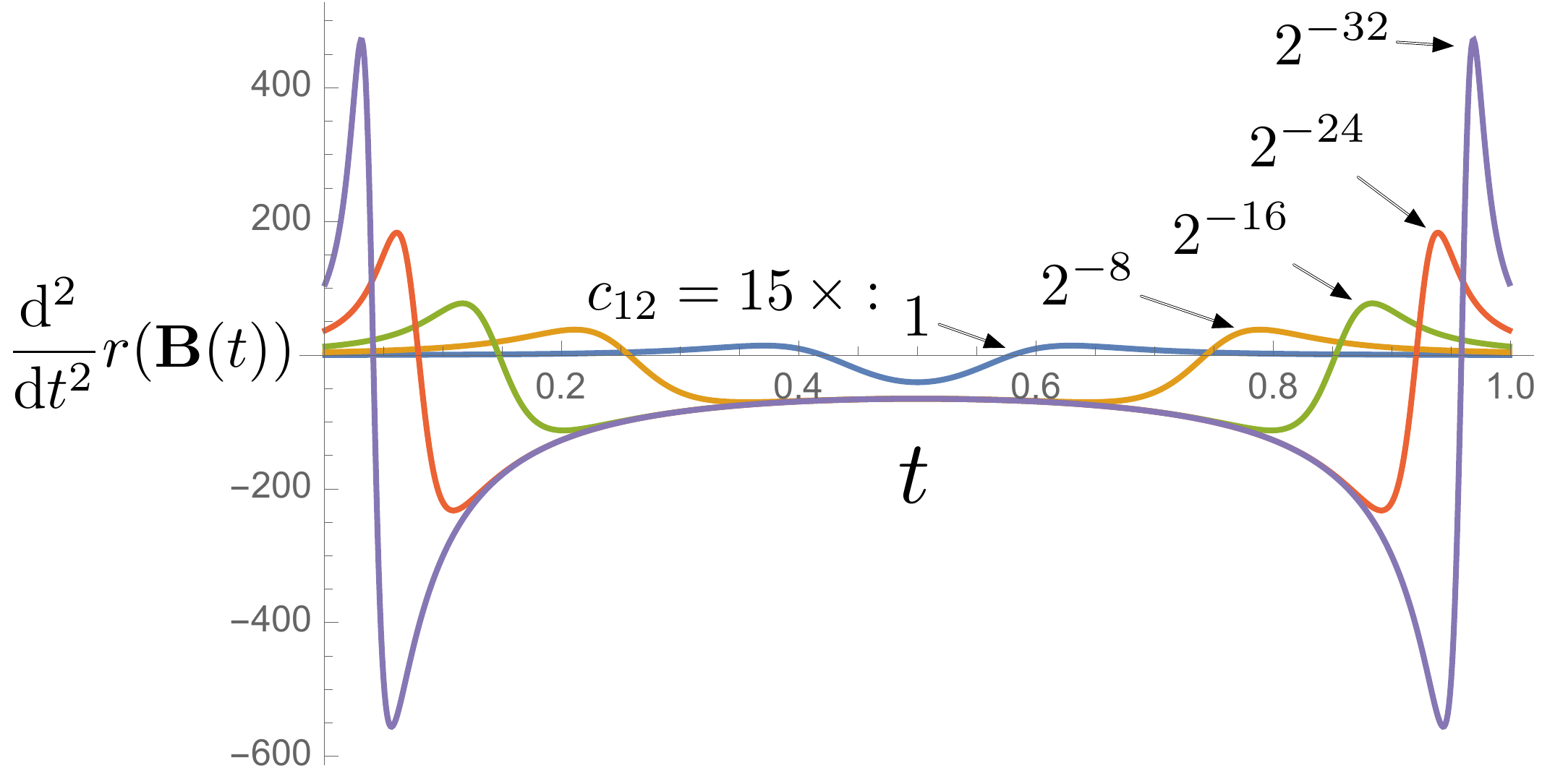}}
\caption {The second derivative of Levinger's function for the cyclic weighted shift matrix from \eqref{eq:CyclicShift16} as weight $c_{12}$ goes toward $0$ by being multiplied by successive powers of $2^{-8}$.}\label{fig:Shift-Matrix_Limit}
\end{figure}

\section{Matrices with Constant Levinger's Function}

Bapat \cite{Bapat:1987:Two} and Fiedler \cite{Fiedler:1995:Numerical} identified matrices with colinear left and right Perron vectors as having constant Levinger's function.  Here we make explicit a property implied by this constraint that appears not to have been described.  We use the centered representation of Levinger's homotopy.  The \emph{symmetric} part of a square matrix $\A$ is 
\an{
\S(\A) \eqdef (\A + \A\tr) / 2 .	\label{eq:SymPart}
}
The \emph{skew symmetric} part of $\A$ is 
\an{
\K(\A) \eqdef (\A - \A\tr)/2 .	\label{eq:SkewPart}
}
Then $\A = \S(\A) + \K(\A)$.   Levinger's homotopy in this centered representation is now, suppressing the $\A$ argument,
\ab{
\C(p) &\eqdef \S + p \, \K, \qquad p \in [-1, 1], 
\stext{and Levinger's function is}
c(p) &\eqdef r((p+1)/2)
= \spr(\S + p \,  \K).
}
The range of $p$ in this centered representation may be extended beyond $[-1,1]$, while maintaining $\C(p) \geq \0$, to the interval $p \in [-\alpha, \alpha]$ where 
\ab{
\alpha = \min_{i,j} \frac{A_{ij} + A_{ji}}{| A_{ji} - A_{ij} | } \geq 1.
}

\begin{Theorem}\label{Theorem:SK}
Let $\A \in \Reals^{n \times n}$ be irreducible and nonnegative.  Then $\spr((1{-}t) \A + t \A\tr)$ is constant in $t \in [0,1]$ if and only if the Perron vector of $\A + \A\tr$ is in the null space of $\A - \A\tr$.
\end{Theorem}
\begin{proof}
\citet{Bapat:1987:Two} and \citet{Fiedler:1995:Numerical} proved that $\spr((1{-}t) \A + t \A\tr)$ is constant in $t \in [0,1]$ if and only if the left and right Perron vectors of $\A$ are colinear.  Suppose the left and right Perron vectors of $\A$ are colinear.  Without loss of generality, they can be normalized to sum to $1$ in which case they are identical.  Let the left and right Perron vectors of $\A$ be $\x$.  
Then
\ab{
\frac{1}{2} (\A + \A\tr) \x &= \spr(\A) \ \x,
\stext{and}
(\A - \A\tr) \x &= \spr(\A)\ (\x - \x) = \0.
}
Hence $\x$ is the Perron vector of $\A + \A\tr$ and $\x > \0$ is in the null space of $\A - \A\tr$.

For the converse, let the Perron vector of $\A + \A\tr$ be $\x > \0$, and let $\x$ be in the null space of $\A - \A\tr$.  Then
\ab{
(\A + \A\tr ) \x &= \spr(\A{+}\A\tr)\ \x
\text{ and }
(\A - \A\tr ) \x = \A \x - \A\tr \x = \0,
}
{which gives}
\ab{
\A \x &
= \frac{1}{2}[(\A + \A\tr ) +( \A - \A\tr)] \x 
= \frac{1}{2}\spr(\A{+}\A\tr)\  \x + \0
= \frac{\spr(\A{+}\A\tr)}{2} \ \x
\stext{and}
\A\tr \x &
= \frac{1}{2}[(\A + \A\tr ) - ( \A - \A\tr)] \x 
= \frac{1}{2}\spr(\A{+}\A\tr)\  \x - \0
= \frac{\spr(\A{+}\A\tr)}{2}\ \x
}
hence $\x$ is a Perron vector of $\A$ and of $\A\tr$.
\end{proof}

\begin{Corollary} 
Let $\S = \S\tr \in \Reals^{n\times n}$ be a nonnegative irreducible symmetric matrix, and $\K = - \K\tr \in \Reals^{n\times n}$ be a nonsingular skew symmetric matrix such that $\A= \S + \K \geq \0$.  Then $n$ is even and $\A$ has a non-constant Levinger's function.
\end{Corollary}
\begin{proof}
If $\K$ is a nonsingular skew symmetric matrix, $n$ must be even, since odd-order skew symmetric matrices are always singular \citep[2-9.27]{Hogben:2014:Handbook}.  If $\C(p) := \S + p \, \K$ with $\K$ nonsingular, then because the null space of $\K$ is $\{\0\}$, $\C(p)$ must have a non-constant Levinger's function $c(p)$ by Theorem \ref{Theorem:SK}.
\end{proof}

The following corollary pursues the observation made by an anonymous reviewer that a matrix $\A$ with colinear left and right Perron vectors is orthogonally similar to a direct sum $\Pmatr{\spr(\A)} \oplus \F$ for some square matrix $\F$.  This entails that the skew symmetric part of $\A$ is orthogonally similar to $\Pmatr{\spr(\A) - \spr(\A)} \oplus (\F-\F\tr) / 2 = \Pmatr{0} \oplus (\F-\F\tr) / 2$, and is thus singular.
\begin{Corollary} 
Let $\S = \S\tr \in \Reals^{n\times n}$ be a nonnegative irreducible symmetric matrix, and $\K = - \K\tr \in \Reals^{n\times n}$ be a skew symmetric matrix, such that $\A = \S + \K \geq \0$.  Let $\Q = (\Q\tr)^{-1}$ be an orthogonal matrix that diagonalizes $\S$ to
\ab{
\Lam &\eqdef \Q\tr \S \Q = \Pmatr
{\spr(\S) & 0 & \cdots & 0\\
0 & \lambda_2 & \cdots & 0 \\
\vdots & \vdots & \ddots & \vdots \\
0 & 0 & \cdots & \lambda_n
} .
}
Then $\A$ has a constant Levinger's function if and only if
\an{\label{eq:KQP}
 \K_1 &\eqdef \Q\tr\K \Q = \Pmatr
{0 & \0\tr \\
\0 &\K_2
} = \Pmatr{0} \oplus \K_2,
}
where $\K_2 = - \K_2\tr \in \Reals^ {n{-}1 \times n{-}1} $ and $\0\tr = (0 \ldots 0) \in \Reals^{n{-}1}$.
\end{Corollary}
\begin{proof}
Since $\S$ is real and symmetric, $\S = \Q \Lam \Q\tr$ is in Jordan canonical form.  Let $\x > \0$ be the normalized Perron vector of $\S$.  Then $\x = [\Q]_1$ is the first column of $\Q$, and the other columns of $\Q$ are orthogonal to $\x$, so $\x\tr \Q = (1\:0 \cdots 0)$.  The necessary and sufficient condition from Theorem \ref{Theorem:SK} for $\A$ to have constant Levinger's function is that $\x\tr\K = \0\tr$, equivalent to
\ab{
\x\tr \K &
= \x\tr \Q\K_1\Q\tr
= (1\: 0 \cdots 0) \K_1\Q\tr
= \0\tr.
}
Since $\Q$ is orthogonal, it has null space $\{\0\}$, so $ (1\: 0 \cdots 0) \K_1\Q\tr = \0\tr$ if and only if $ (1\: 0 \cdots 0) \K_1 = \0\tr$, which is the top row of $\K_1$.  $\K_1$ and $\K_2$ must be skew symmetric since $\K$ is skew symmetric, as can be seen immediately from transposition.  The skew symmetry of $\K_1$ implies its first column must also be all zeros as its first row is, establishing the form given in \eqref{eq:KQP}.
\end{proof} 

\section{Conclusions}

We have shown that it is not in general true that the spectral radius along a line from a nonnegative square matrix $\A$ to its transpose --- Levinger's function --- is concave.  Our counterexamples to concavity have a simple principle in the case of direct sums of block matrices, namely, that the maximum of two concave functions need not be concave.  However, for the other examples we present --- Toeplitz matrices, and positive circuit or cyclic weighted shift matrices --- whatever principles underly the nonconcavity remain to be discerned.  Also remaining to be discerned are the properties of matrix families --- a few of which we have presented here --- that guarantee concave Levinger functions.  A general characterization of the range of $t$ for which the spectral radius is concave in Levinger's homotopy remains an open problem.

\section*{Biographical Note}

Bernard W. Levinger (Berlin, Germany, September 3, 1928 -- Fort Collins, Colorado, USA, January 17, 2020) and his family fled Nazi Germany to England in 1936, to Mexico in 1940, and to the United States in 1941, which initially placed them in an immigration prison and deported them to Mexico, but which ultimately allowed their immigration, whereupon they settled in New York City.  Levinger graduated from Bronx High School of Science and earned a doctorate in mathematics from New York University. He was Professor of Mathematics and Professor Emeritus at Colorado State University, Fort Collins. He leaves a large family, including his wife Lory of more than 65 years.\citep{Levinger:1928-2020}

\section*{Acknowledgements}
L.A. thanks Marcus W. Feldman for support from the Stanford Center for Computational, Evolutionary and Human Genomics and the Morrison Institute for Population and Resources Studies, Stanford University; the Mathematical Biosciences Institute at The Ohio State University, for its support through U.S. National Science Foundation awards DMS-0931642 and DMS-1839810, ``A Summit on New Interdisciplinary Research Directions on the Rules of Life'';  and the Foundational Questions Institute and Fetzer Franklin Fund, a donor advised fund of Silicon Valley Community Foundation, for FQXi Grant number FQXi-RFP-IPW-1913.  J.E.C. thanks the U.S. National Science Foundation for grant DMS-1225529 during the initial phase of this work and Roseanne K. Benjamin for help during this work.

\bibliographystyle{model2-names.bst}

\end{document}